\documentclass[12pt]{article}
\usepackage{amsthm,amsfonts,amsmath, amscd,dsfont}
\usepackage{mathrsfs}

\swapnumbers
                              
\theoremstyle{plain}
\newtheorem{thm}{THEOREM}[section]
\newtheorem{lm}[thm]{LEMMA}

\theoremstyle{definition}

\theoremstyle{remark}
\newcommand{\upchi}{\raise1pt\hbox{$\chi$}}

\newcommand{\C}{{\mathord{\mathbb C}}}

\newcommand{\tr}{{\rm Tr}}
\numberwithin{equation}{section}
\newcommand{\un}{{\rm 1\kern -2.5pt l}}

\begin{document}
%%%%%%%%DRAFT%%%%%%%
\iffalse
\fi
%%%%%%%%%%%%%%%%%%%%%%

\title{{\sc  A monotonicity version of a concavity theorem of Lieb}}
\author{
\vspace{5pt}  Eric A. Carlen$$   \\
\vspace{5pt}\small{Department of Mathematics, Hill Center,}\\[-6pt]
\small{Rutgers University,
110 Frelinghuysen Road
Piscataway NJ 08854-8019 USA}\\
 }

\footnotetext[1]{Work partially supported by U.S.
National Science Foundation grant  DMS 2055282.}

\medskip

\maketitle

\begin{abstract} We give a simple proof of a strengthened version of a theorem of Lieb that played a key role in the proof of strong subadditivity of the quantum entropy.
\end{abstract}

\section{Introduction}   
We write $M_n(\C)$ to denote the $n\times n$ complex matrices, $M_n^{+}(\C)$ the subset consisting of positive semidefinite matrices, and $M_n^{++}(\C)$ the 
subset consisting of positive definite matrices.  The following theorem was proved by Lieb in 1973 \cite[Theorem 6]{L73}.
\begin{thm}[Lieb]\label{L4} 
For all self-adjoint $H\in M_n(\C)$ the function
\begin{equation}\label{L4con}
Y \mapsto \tr\left[\exp(H + \log Y)\right]
\end{equation}
is concave on  $M_n^{++}(\C)$. 
\end{thm}
As a simple consequence of this, Lieb deduced his {\em triple matrix inequality}, a generalization of the Golden-Thompson inequality to 
three self adjoint matrices. This played a fundamental role in the proof of  strong subadditivity of the quantum entropy \cite{LR}. For more recent applications of Theorem~\ref{L4}, see  the influential paper of Tropp \cite{Tr12b}.

We now prove a stronger version of this theorem in terms of monotonicity instead of concavity. 
Recall that a linear map $\Phi: M_n(\C)\to M_m(\C)$ is  {\em positive} if $\Phi(A)\in M_m^+(\C)$ whenever $A\in M_n^+(\C)$, and is unital if $\Phi(I) =I$, and 
{\em trace preserving} if $\tr[\Phi(X)] = \tr[X]$ for all $X\in M_n(\C)$.  We equip $M_n(\C)$ with the Hilbert-Schmidt inner product, and we use $\Phi^\dagger$ to denote the corresponding adjoint of a linear map on $M_n(\C)$. Note that $\Phi$ is unital if and only if $\Phi^\dagger$ is trace preserving. Our main result is:

\begin{thm}\label{newmon} Let $\Phi:M_n(\C)\to M_m(\C)$ be unital and positive. Then for all self-adjoint $H\in M_n(\C)$ and all $Y\in M_m^{++}(\C)$, 
\begin{equation}\label{L4mon}
 \tr\left[\exp(H + \log \Phi^\dagger(Y))\right]  \geq \tr\left[\exp(\Phi(H) + \log Y)\right] \ .
\end{equation}
\end{thm}

Before giving the very simple proof, we explain how Theorem~\ref{newmon} implies Theorem~\ref{L4}.   Let $\Phi: M_n(\C) \to M_{2n}(\C)$ be defined by 
\begin{equation}\label{phiex}
\Phi(H) = \left[\begin{array}{cc} H & 0\\ 0 & H\end{array}\right] \quad{\rm and\ hence}\quad 
\Phi^\dagger\left(\left[\begin{array}{cc} Y_{1} & 0\\0 & Y_{2}\end{array}\right]\right) = Y_{1}+Y_{2}\ . 
\end{equation}
 Thus,
$$\tr[e^{H+\log (Y_1+Y_2)}]  \geq  \tr[e^{H+\log Y_1}]  + \tr[e^{H+\log Y_2}]\ .$$
Since $Y \mapsto  \tr[e^{H+\log Y}]$ is homogeneous of degree one, this is the same as concavity. 
Note that this particular map $\Phi$ is not only positive; it is completely positive. 

Our proof is based on the  well-known and elementary Gibbs variational principle for the free energy in terms of the entropy $S(X) = -\tr[X\log X]$  of a density matrix $X$. This states that  for all self-adjoint $K\in M_n(\C)$
\begin{equation}\label{gibbs2}
\log(\tr[e^{K}]) =  \sup\{ \tr[X K] - \tr[X \log X]\ :\ X\in M_n^{++}(\C)\ , \tr[X] =1\ \}\ .
\end{equation}
A short  proof from scratch can be found in Appendix  A of \cite{CL19}.  There is a simple variant involving the relative entropy 
$$D(X||Y) = \tr[X(\log X - \log Y)]$$
of two density matrices:
For $W \in M_n^{++}(\C)$ replace $K$ with $K + \log W$ in \eqref{gibbs2} to conclude that
 for all self-adjoint $K\in M_n(\C)$ and all $W\in M_m^{++}(\C)$, 
\begin{equation}\label{gibbs4}
\log(\tr[e^{K+\log W}]) =  \sup\{ \tr[XK] - D(X||W)\ :\ X\in M_n^{++}(\C)\ , \tr[X] =1\ \}\ .
\end{equation}

We shall also use the result due to M\"uller-Hermes and Reeb \cite{MHR17} that the relative entropy is monotone under positive trace-preserving maps 
$\Phi^\dagger$. 
That is, for all such maps $\Phi^\dagger$, and all density matrices $X,Y$
\begin{equation}\label{MHR}
D(\Phi^\dagger(X)||\Phi^\dagger(Y)) \leq D(X||Y)\ .
\end{equation}
A somewhat weaker result is due to Uhlmann, who proved in 1977 that \eqref{MHR} is true whenever 
$\Phi^\dagger$ is the dual of a unital {\it Schwarz map}; i.e., a unital map $\Phi$ such that $\Phi(A^*A) \geq \Phi(A)^*\Phi(A)$. Earlier still in 1973, 
Lindblad had shown that \eqref{MHR} was valid unital completely positive  $\Phi$.  Lindblad's proof relied on the Lieb Concavity Theorem \cite{L73}.
 It is well known that every completely positive unital map is  a Schwarz map, and it is evident that Schwarz maps are positive. 
By now, quite simple proofs of the monotonicity of the relative entropy under the dual of a unital Schwarz map are known; see e.g., \cite{P03}.  The deep result in \cite{MHR17} turns on ideas introduced by Beigi \cite{B13}.

\begin{proof}[Proof of Theorem~\ref{newmon}] By \eqref{gibbs4},
\begin{eqnarray*} 
\log(\tr[e^{H+\log \Phi^\dagger(Y)}]) &=&  \sup\{ \tr[X H ] - D(X|| \Phi^\dagger(Y))\ :\ X\in M_n^{++}(\C)\ , \tr[X] =1\ \}\\
&\geq&  \sup\{ \tr[\Phi^\dagger(W) H ] - D(\Phi^\dagger(W)|| \Phi^\dagger(Y))\ :\ W\in M_m^{++}(\C)\ , \tr[W] =1\ \}\\
&\geq&  \sup\{ \tr[W \Phi(H) ] - D(W || Y)\ :\ W\in M_m^{++}(\C)\ , \tr[W] =1\ \}\\
&=& \log \tr\left[e^{\Phi(H) + \log Y}\right]  \ .
\end{eqnarray*}
where we used $\{\  \Phi^\dagger(W) \ :\  W\in M_m^{++}(\C)\ , \tr[W] =1\ \} \subset \{\  X\in M_n^{++}(\C)\ , \tr[X] =1\ \}$ and \eqref{MHR}.  
Exponentiating both sides of the inequality yields \eqref{L4mon}.
\end{proof} 

While Theorem~\ref{newmon} is strictly stronger than Theorem~\ref{L4}, the greatest interest in it may lie in the very simple proof that its proof provides of Theorem~\ref{L4}.  This is an interesting example of how it may be easiest to prove a concavity result by first proving a monotonicity result, and then applying that to the particular map $\Phi$ that is defined in \eqref{phiex}. 

It is interesting to observe another advantage of the monotonicity approach to convexity or concavity inequalities.   A duality method for proving convexity and concavity inequalities was introduced by myself and Lieb in \cite{CL08} which  uses the lemma:

\begin{lm}\label{rock}
If $f(x,y)$ is jointly convex in $x,y$,  then $g(x) := \inf_y f(x,y)$ is convex.  If $f(x,y)$ is jointly concave in $x,y$,  then $g(x) := \sup_y f(x,y)$ is  concave.  
 \end{lm}
 
This may be found in \cite[Theorem 1]{R74}, and the simple proof is also given in \cite{CL08}.  Since $(X,W) \mapsto D( X || W)$, is jointly convex,  
for fixed $K$,
$(X,W) \mapsto \tr[XK] - D( X || W)$, is jointly concave. Then by Lemma~\ref{rock} and \eqref{gibbs4}, $W \mapsto \log \tr[ e^{K+ \log W}]$ is concave.   
However, this is a weaker statement than Theorem~\ref{L4} since if $g(x) = \log(f(x))$ with $f$ positive and twice continuously differentiable. 
$$g''(x) = -\left(\frac{f'(x)}{f(x)}\right)^2 +\frac{1}{f(x)}f''(x)\ .$$
Thus, concavity of $f$ implies concavity of $\log f$, but not the  other way around.  However, monotonicity of $f$ is equivalent to monotonicity of $\log f$. 
For this reason, we could simply use the Gibbs variational principle to prove our Theorem~\ref{newmon}.  Tropp \cite{Tr12} found an ingenious variational representation of  $W \mapsto \tr[ e^{K+ \log W}]$ which allowed him to give a proof of Theorem~\ref{L4}  using the joint convexity of the relative entropy and Lemma~\ref{rock}, along the lines of \cite{CL08}. However,  when using duality to to prove the monotonicity theorem, the logarithm is not an issue because log monotonicity is the same as monotonicity, and we do not need Lemma~\ref{rock}.  Already in 1973 Epstein \cite{E73} gave a second proof of Theorem~\ref{L4} using the theory of Herglotz functions, but this is considerably more involved than the present proof. For more infromation, see \cite{C22}. 

\medskip
\noindent{\bf Aclknowledgement} I thank Aleksander  M\"uller-Hermes for helpful correspondence on his work with Reeb.

\end{document}